\documentclass[11pt]{amsart}
\usepackage{amssymb,amsfonts}
\usepackage{amsmath,amscd}
\usepackage[all,arc]{xy}
\usepackage{enumerate}
\usepackage{mathrsfs}
\usepackage[pdftex]{graphicx}
\usepackage[utf8]{inputenc}
\usepackage[T1]{fontenc}
\usepackage[usenames,dvipsnames]{color}
\usepackage[bookmarks=false]{hyperref}
\usepackage{dsfont}
\usepackage{tikz}
\usetikzlibrary{automata,positioning}

\theoremstyle{plain}
\newtheorem{thm}{Theorem}[section]
\newtheorem{cor}[thm]{Corollary}
\newtheorem{prop}[thm]{Proposition}
\newtheorem{lem}[thm]{Lemma}

\newenvironment{customThm2}[1]{\paragraph*{\textbf{ #1.}}\itshape}{\par}

\theoremstyle{definition}

\theoremstyle{remark}
\newtheorem{rmk}[thm]{Remark}

\newcommand{\bbH}{\mathbb{H}} %hyperbolic space
\newcommand{\bbR}{\mathbb{R}} %reals
\newcommand{\bbS}{\mathbb{S}}

\newcommand*{\defeq}{\mathrel{\vcenter{\baselineskip0.5ex \lineskiplimit0pt
			\hbox{\scriptsize.}\hbox{\scriptsize.}}}%
	=}

\makeatletter
\let\c@equation\c@thm
\makeatother
\numberwithin{equation}{section}

\title{On the displacement of generators of free Fuchsian groups}

\author{Yan Mary He}
\address{Department of Mathematics\\
  University of Chicago\\
  Chicago, IL 60637}
\email{he@math.uchicago.edu}
\date{\today}

\begin{document}

\maketitle
%\tableofcontents

\begin{abstract}
We prove an inequality that must be satisfied by displacement of generators of free Fuchsian groups, which is the two-dimensional version of the $\log (2k-1)$ Theorem for Kleinian groups due to Anderson-Canary-Culler-Shalen (\cite{ACCS}). As applications, we obtain quantitative results on the geometry of hyperbolic surfaces such as the two-dimensional Margulis constant and lengths of a pair of based loops, which improves a result of Buser's.
\end{abstract}

\section{Introduction}
In \cite{ACCS}, Anderson, Canary, Culler and Shalen proved the following remarkable theorem which gives a displacement constraint for the generators of a geometrically finite free Kleinian group.

\begin{thm}[\cite{ACCS}, Theorem 6.1, $\log(2k-1)$ Theorem]
Let $k \ge 2$ be an integer and let $\Phi$ be a geometrically finite Kleinian group freely generated by $\{g_1, \cdots, g_k \}$. Let $z$ be any point in $\bbH^3$ and denote by $d_i = \text{d}(z, g_iz)$ the hyperbolic distance between $z$ and $g_iz$. Then we have
\begin{align}\label{log2k_ineq}
\sum_{i=1}^k \frac{1}{1+e^{d_i}} \le \frac{1}{2}.
\end{align}
In particular, there exists an $i \in \{1, \cdots, k\}$ such that $d_i \ge \log\left(2k-1\right).$
\end{thm}

The $\log(2k-1)$ Theorem has many interesting applications in relating quantitative geometry, especially volume estimates, of hyperbolic $3$-manifolds and their classical topological invariants. Using topological arguments, one can show that the fundamental group of a hyperbolic $3$-manifold contains many free subgroups. The constraints on these free subgroups consequently impose constraints on the geometry of the manifold, which can be used to obtain strong volume estimates (see \cite{ACCS} and \cite{CS} for more details). 

In this paper, we study the displacement constraint problem in two dimensions. More specifically, we obtain an inequality analogous to (\ref{log2k_ineq}) for free Fuchsian groups.

\vspace{0.3cm}
\begin{customThm2}{Main Theorem}
Let $k \ge 2$ be an integer and let $\Phi$ be a free Fuchsian group generated by $\{g_1 \cdots g_k \}$. Let $z$ be any point in $\bbH^2$ and denote by $d_i = \text{d}(z, g_iz)$ the hyperbolic distance between $z$ and $g_iz$. Then we have
\begin{align}\label{ineq_intro}
	\sum_{i=1}^k \arccos \left(\tanh\left(\frac{d_i}{2}\right)\right)
	&\le \frac{\pi}{2}.
\end{align}
In particular, there exists an $i \in \{1,\cdots,k\}$ such that $$d_i \ge \log\left(\dfrac{1+\cos\frac{\pi}{2k}}{1-\cos\frac{\pi}{2k}}\right).$$ When $k=2$, the upper bound is realized by $\Phi = \Gamma(2) = \left\langle \left( \begin{array}{cc}
1 & 2 \\
0 & 1 \end{array} \right), \left( \begin{array}{cc}
1 & 0 \\
2 & 1 \end{array} \right) \right\rangle$ and $z = \sqrt{-1}$.
\end{customThm2}
\vspace{0.3cm}

Although a free Fuchsian group is also a geometrically finite Kleinian group, our theorem is strictly stronger than the $\log (2k-1)$ Theorem. An easy way to see this is to observe that, when $k=2$, $\log(3+2\sqrt{2})$ is greater than $\log 3$ and it is realized. Rigorous calculations are carried out in Remark \ref{rmk_compare}.

The Main Theorem can be applied to study quantitative geometry of hyperbolic surfaces. Recall that, if $M$ is a closed orientable hyperbolic $n$-manifold, a positive real number $\varepsilon$ is called a {\it Margulis number} for $M$ if for any $z \in \bbH^n$ and $\alpha, \beta \in \pi_1M$, $\text{max}\{d(z,\alpha(z)), d(z,\beta(z))\} < \varepsilon$ implies that $\alpha$ and $\beta$ commute. The Margulis Lemma (\cite{BP}, Chapter D) states that for every $n \ge 2$ there is a positive constant $\varepsilon_n$ called the {\it Margulis constant} which serves as a Maruglis number for every closed orientable hyperbolic $n$-manifold. It follows from the Main Theorem that $\log\left(3+2\sqrt{2}\right)$ is the $2$-dimensional Margulis constant. This is consistent with the number found by Yamada in \cite{Yamada} using the signatures of non-elementary torsion-free $2$-generator Fuchsian groups.

The Main Theorem also imposes constraints on the lengths of a pair of based loops on a hyperbolic surface. A classical result of Buser's (\cite{Bu}) states that if $\alpha_1$ and $\alpha_2$ are two intersecting simple closed geodesics on a compact hyperbolic surface, then their lengths $\ell_1$ and $\ell_2$ must satisfy $\sinh(\ell_1/2) \cdot \sinh(\ell_2/2) \ge 1$. Using the Main Theorem, one can show that the same inequality holds but $\alpha_1$ and $\alpha_2$ can be promoted to be a pair of based loops as long as their (free) homotopy classes generate a free group of rank $2$.

\subsection{Strategy of the proof}
We first prove a special case of the Main Theorem where we assume that the quotient surface $\Sigma_{\Phi}$ has finite area (Theorem \ref{finite_area_thm}). The proof is based on Culler-Shalen's paradoxical decomposition of the Patterson-Sullivan measure. 

If $\Phi$ is a free group generated by $\{g_1, \cdots, g_k\}$, then $\Phi$ admits a decomposition into a singleton $\{Id\}$ and $2k$ disjoint subsets, each one of which consists of reduced words with a given first letter coming from the symmetric generating set. This decomposition has the property that each of the $2k$ subsets is mapped to the complement of another by left multiplying a generator or the inverse of a generator.

Pick a point $z$ in $\bbH^2$ and identify the orbit $\Phi \cdot z$ with $\Phi$. Since the Patterson-Sullivan measure $\mu$ associated to $\Phi$ is constructed as the limit of measures supported on the orbit, it respects the combinatorial structure of the group; that is, there exists a decomposition of $\mu$ into $2k$ measures, one for each element in the symmetric generating set, such that each one of the measures is transformed into the complement of another by a generator or its inverse (Proposition \ref{paradecomp_prop}).

If the quotient surface $\Sigma_{\Phi}$ of a Fuchsian group $\Phi$ has finite area, then every positive $\Phi$-invariant superharmonic function on $\bbH^2$ is constant. Then one can show (Proposition \ref{prop3.9}) that the Patterson-Sullivan measure $\mu$ associated to $\Phi$ is in fact the arc length measure on the boundary circle at infinity $\bbS_{\infty}^1$. 
Therefore, we obtain a decomposition of the arc length measure into $2k$ measures. These decomposed measures give an estimate for the displacement of each generator (Lemma \ref{estimates_lem}), which leads to the inequality involving displacements for all generators.

To prove the theorem when $\Sigma_{\Phi}$ has infinite area, we prove a simple but seemingly surprising fact (Lemma \ref{lem_holemb}) that $\Sigma_{\Phi}$ admits a holomorphic embedding, which is also a homotopy equivalence, into a finite area surface $S_{\infty}$. To this end, we first show that each funnel end of $\Sigma_{\Phi}$ is conformal to an annulus. Therefore, $\Sigma_{\Phi}$ is conformal to the Riemann surface $S$ obtained by gluing an annulus to each boundary geodesic of its convex core $C(\Sigma_{\Phi})$. Moreover, $S$ is holomorphically embedded in the finite-area surface $S_{\infty}$ obtained by gluing a half infinite cylinder $S^1 \times \bbR^+$ to each boundary geodesic of $C(\Sigma_{\Phi})$. The theorem then follows as holomorphic maps are distance non-increasing, the function $\arccos \tanh (x/2)$ is strictly decreasing and the inequality holds on $S_{\infty}$.

\subsection{Organization of the paper}
The paper is organized as follows. In Section 2, we give a very brief overview of conformal densities and Culler-Shalen's paradoxical decompositions, which we will need to prove the finite co-area case of the Main Theorem in Section 3. Section 4 is devoted to proving the Main Theorem and we discuss its geometric applications in Section \ref{section_app}.

\subsection{Acknowledgments} I would like to thank Peter Shalen for his invaluable advice, gracious support and encouragement throughout this project. I thank David Dumas for pointing out the proof of Lemma \ref{lem_holemb}, Danny Calegari for helpful conversations and Yong Hou for commenting on the draft.

% % % % % % % % % % % %% % % % % % % % % % % %% % % % % % % % % % % %% % % % % % % % % % % %% % % % % % % % % % % %% % % % % % % % % % % %

\section{Conformal densities and paradoxical decompositions}
In this section, we give an expository account of conformal densities and Culler-Shalen's paradoxical decompositions of the Patterson-Sullivan measure. Interested readers are refered to \cite{CS} for more details. 

\subsection{Conformal densities}
Recall that the Poisson kernel $P: \bbH^n \times \bbH^n \times S_{\infty}^{n-1} \to \bbR$ is given by
$$P(z,z',\zeta) \defeq (\cosh d(z,z') - \sinh(d(z,z')\cos\angle z'z\zeta)^{-1}$$ where $d$ denotes the hyperbolic distance.

Let $n \ge 2$ be an integer and $D \in [0,n-1]$. A {\it $D$-conformal density for $\bbH^n$} is a family $\mathscr{M} = (\mu_z)_{z \in \bbH^n}$ of finite Borel measures on $S_{\infty}^{n-1}$ such that $d\mu_{z'} = P(z,z',\cdot)^D d\mu_z$ for any $z, z' \in \bbH^n$. For example, the {\it area density} $\mathscr{A} = (A_z)_{z\in \bbH^n}$ is an $(n-1)$-conformal density, where $A_z$ is the normalized area measure on $S_{\infty}^{n-1}$ induced by the round metric centered at $z$. 

Let $\Gamma$ be a group of isometries of $\bbH^n$. A conformal density $\mathscr{M} = (\mu_z)_{z \in \bbH^n}$ is {\it $\Gamma$-invariant} if $\gamma^*_{\infty}\mu_{\gamma z} = \mu_z$, for all $z\in \bbH^n$ and $\gamma \in \Gamma$. Clearly, the area density is $\Gamma$-invariant.

Given a $D$-conformal density $\mathscr{M} = (\mu_z)_{z \in \bbH^n}$, we define a nonnegative function $u_{\mathscr{M}}: \bbH^n \to [0,1]$ by $u_{\mathscr{M}}(z) = \mu_z(S_{\infty}^{n-1})$. In case of the area density $\mathscr{A}$, $u_{\mathscr{A}} \equiv 1$. It can be shown that the function $u_{\mathscr{M}}(z)$ is smooth and satisfies the equation $$\Delta u_{\mathscr{M}} = -D(n-D-1)u_{\mathscr{M}}$$ where $\Delta$ is the (hyperbolic) Laplacian on $\bbH^n$. In particular, $u_{\mathscr{M}}$ is superharmonic. Moreover, if $\mathscr{M}$ is $\Gamma$-invariant for some group $\Gamma$ of isometries of $\bbH^n$, then $u_{\mathscr{M}}$ is also $\Gamma$-invariant.

The next proposition asserts the uniqueness of a $\Gamma$-invariant conformal density for $\bbH^n$ when all $\Gamma$-invariant positive superharmonic functions are constant.
\begin{prop}[\cite{CS}, Proposition 3.9]\label{prop3.9}
	Let $\Gamma$ be a non-elementary discrete group of isometries of $\bbH^n$. Suppose every $\Gamma$-invariant positive superharmonic function on $\bbH^n$ is constant. Then any $\Gamma$-invariant conformal density for $\bbH^n$ is a constant multiple of the area density.
\end{prop}

For the convenience of the reader, we sketch the proof here.
\begin{proof}
	A $\Gamma$-invariant conformal density gives a $\Gamma$-invariant positive superharmonic function on $\bbH^n$ by integrating the Poisson kernel against the measure. Since $u_{\mathscr{M}}(z) \equiv C$, which is a constant, $u_{\mathscr{M}} = u_{C\mathscr{A}}$, where $\mathscr{A}$ is the area density. Moreover, it can be shown that any $(n-1)$-conformal density is determined by the associated harmonic function. Hence, $\mathscr{M} = C\mathscr{A}$.
\end{proof}

\subsection{Culler-Shalen's paradoxical decompositions}
A free group on $k$ generators can be decomposed into $2k+1$ disjoint sets depending on the first letter of the group element. We show in this subsection (Proposition \ref{paradecomp_prop}) that there exists a decomposition of the Patterson-Sullivan measure which respects this decomposition of the group. The proof is based on a generalized Patterson's construction of conformal densities for uniformly discrete subsets of $\bbH^n$.

A subset $W$ of $\bbH^n$ is {\it uniformly discrete} if there exists $\delta>0$ such that $d(z,w) > \delta$ for all $z,w \in W$. In particular, any orbit of a discrete group of isometries of $\bbH^n$ is uniformly discrete.

\begin{lem} [\cite{CS}, Proposition 4.2]\label{lem4.2}
Let $W$ be an infinite uniformly discrete subset of $\bbH^n$. Let $\mathcal{B}$ be a countable collection of subsets of $W$ which contains $W$. Then there exists a $D \in [0,n-1]$ and a family $(\mathscr{M}_V )_{V \in \mathcal{B}}$ of $D$-conformal densities satisfying the following conditions
\begin{enumerate}
\item $\mathscr{M}_W \neq 0.$
\item For any finite family $\left(V_i\right)_{i=1}^m$ of disjoint sets in $\mathcal{B}$ such that $V = \displaystyle\amalg_{i=1}^m V_i \in \mathcal{B}$, we have $\mathscr{M}_V = \sum_{i=1}^{m}\mathscr{M}_{V_i}$.
\item For any $V \in \mathcal{B}$ and any isometry $\gamma$ of $\bbH^n$ with $\gamma V \in \mathcal{B}$, we have $\gamma_{\infty}^*(\mathscr{M}_{\gamma V}) = \mathscr{M}_V$.
\item For any $V \in \mathcal{B}$, supp$(\mathscr{M}_V)$ is contained in the limit set of $V$. In particular, for any finite set $V \in \mathcal{B}$, we have $\mathscr{M}_V = 0$.
\end{enumerate}
\end{lem}

Note that by taking $\mathcal{B} = \{W = \text{any orbit of } \Gamma \}$, we obtain the Patterson-Sullivan density which is $\Gamma$-invariant and supported on the limit set of $\Gamma$.

\begin{prop}\label{paradecomp_prop}
	Let $\Gamma$ be a free Fuchsian group with a symmetric generating set $\Psi = \{g_1^{\pm} \cdots g_k^{\pm}\}$. Let $z_0 \in \bbH^2$. Then there exist a number $D \in [0,1]$, a $\Gamma$-invariant $D$-conformal density $\mathscr{M}=(\mu_{z_0})$ for $\bbH^2$ and a family $(\nu_{\psi})_{\psi \in \Psi}$ of Borel measures on $S_{\infty}^1$ such that
	\begin{enumerate}
		\item $\mu_{z_0}(S_{\infty}^1)=1$,
		\item $\mu_{z_0} = \sum_{\psi \in \Psi} \nu_{\psi}$, and
		\item for each $\psi \in \Psi$ we have $$\int_{S_{\infty}^{1}} P(z_0,\psi^{-1}z_0,\zeta)^D d\nu_{\psi^{-1}}(\zeta) = 1 - \int_{S_{\infty}^{1}} d\nu_{\psi}(\zeta).$$
	\end{enumerate} 
\end{prop}

\begin{proof}
Let $W = \Gamma z_0$ be an orbit of $\Gamma$. Then $$W = \{z_0\} \amalg \left(\amalg_{\psi \in \Psi} V_{\psi}\right)$$ where $V_{\psi} = \{\gamma z_0: \gamma \in \Gamma \text{ starts with } \psi\}$. Let $\mathcal{B}$ be a countable collection of subsets of $W$ which has the form $\{z_0\} \amalg (\amalg_{\psi \in \Psi'}V_{\psi})$ or $\amalg_{\psi \in \Psi'}V_{\psi}$ where $\Psi' \subset \Psi$.

Apply Lemma \ref{lem4.2} and set $\mathscr{M} = \mathscr{M}_W$, $\mu_{z_0} = \mu_{W, z_0} \in \mathscr{M}$ and $\nu_{\psi} = \mu_{V_{\psi},z_0} \in \mathscr{M}_{V_{\psi}}$. Then
$$\mu_{z_0} = \mu_{z_0,z_0}+\sum_{\psi \in \Psi}\mu_{V_{\psi},z_0} = 0 + \sum_{\psi \in \Psi}\nu_{\psi}.$$

Since $\mathscr{M}_{V_{\psi^{-1}}} = \psi^*_{\infty}(\mathscr{M}_{W-V_{\psi}}) = \psi^*_{\infty}(\mathscr{M} - \mathscr{M}_{V_{\psi}})$, 
$$\mu_{V_{\psi^{-1}}, \psi(z_0)} = \psi^*_{\infty}(\mu_{z_0} - \nu_{\psi}).$$

On the other hand, $d\mu_{V_{\psi^{-1}}, \psi(z_0)} = P(z_0,\psi^{-1}z_0,\cdot)^D d\mu_{V_{\psi^{-1}}}$ since $\mathscr{M}_{V_{\psi^{-1}}}$ is a $D$-conformal density.
Hence, $$\int_{S_{\infty}^{1}} P(z_0,\psi^{-1}z_0,\cdot)^D d\nu_{\psi^{-1}} = \int_{S_{\infty}^{1}} d(\mu_{z_0} - \nu_{\psi}) = 1 - \int_{S_{\infty}^{1}} d\nu_{\psi}.$$
\end{proof}

\section{The finite co-area case}
This section is devoted to proving the main theorem for free Fuchsian groups of rank $k \ge 2$ with finite-area quotient surfaces. The foundation of the proof is the Culler-Shalen's paradoxical decompositions of the Patterson-Sullivan measure developed in the previous section. For each generator, these decomposed measures give rise to a good lower bound for the displacement (Lemma \ref{estimates_lem}), which allows us to deduce the inequality satisfied by displacement of all the generators.

\begin{thm}\label{finite_area_thm}
Let $k \ge 2$ be an integer and let $\Phi$ be a finite co-area free Fuchsian group generated by $\{g_1 \cdots g_k \}$. Let $z$ be any point in $\bbH^2$ and denote by $d_i = \text{d}(z, g_iz)$ the hyperbolic distance between $z$ and $g_iz$. Then we have
\begin{align}
	\sum_{i=1}^k \arccos \left(\tanh\left(\frac{d_i}{2}\right)\right)
	&\le \frac{\pi}{2}.
\end{align}
In particular, there exists an $i \in \{1,\cdots,k\}$ such that $$d_i \ge \log\left(\frac{1+\cos\frac{\pi}{2k}}{1-\cos\frac{\pi}{2k}}\right).$$
\end{thm}

\begin{proof}
According to Proposition \ref{paradecomp_prop}, there exist a number $D \in [0,1]$ and Borel measures $\mu_{z}$, $\nu_{g_i}$ and $\nu_{{g_i}^{-1}}$ satisfying conditions (1)--(3) in the proposition. Also, since $\Sigma_{\Phi}$ has finite area, it does not admit any non-constant positive superharmonic functions. Hence, by Proposition \ref{prop3.9}, the $D$-conformal density is a constant multiple of the area density which is a $1$-conformal density, i.e. $D=1$.

Let $\alpha_i = \nu_{g_i}(S_{\infty}^1)$ and $\beta_i = \nu_{{g_i}^{-1}}(S_{\infty}^1)$, for $i = 1, \cdots, k$. Without loss of generality, we assume that $\alpha_i \le \beta_i$ so that $\alpha_i \in (0, 1/2)$ and $\beta_i \in (0, 1)$.

For now, let us assume the following lemma.
\begin{lem}\label{estimates_lem}
	Let $a$ and $b$ be numbers such that $0 \le a \le 1/2$ and $0 \le b \le 1$, let $\gamma$ be an isometry of $\bbH^2$ and let $z$ be a point in $\bbH^2$. Suppose that $\nu$ is a measure on $S^1_{\infty}$ such that
	\begin{enumerate}
		\item $\nu \le A_z$
		\item $\nu(S^1_{\infty}) \le a$, and
		\item $\displaystyle\int_{S^1_{\infty}} P(z,\gamma^{-1}z,\cdot)d\nu \ge b$. 
	\end{enumerate}
	Then $$\text{d}(z, \gamma z) \ge \log \left(\frac{\tan (b\pi/2)}{\tan (a\pi/2)}\right).$$
\end{lem}

Applying the lemma with $a = \alpha_i$ and $b = 1-\beta_i$, we obtain
\begin{align*}
e^{d_i} \ge \frac{\tan(\pi (1-\beta_i)/2)}{\tan(\pi \alpha_i/2)} &= \frac{\cot(\pi \beta_i/2)}{\tan(\pi \alpha_i/2)}\\
&= \frac{\cos(\pi \beta_i/2) \cos(\pi \alpha_i/2)}{\sin(\pi \beta_i/2) \sin(\pi \alpha_i/2)}\\
&= \frac{\cos(\pi p_i)}{\sin(\pi \beta_i/2) \sin(\pi \alpha_i/2)}+1  \text{~where~} p_i = \frac{\alpha_i+\beta_i}{2}\in (0,1/2).
\end{align*}

Note that $\sin(x)\sin(y) \le \left(\sin \left(\dfrac{x+y}{2}\right)\right)^2 = \dfrac{1-\cos(x+y)}{2}$. \\Hence,
\begin{align*}
e^{d_i}+1 &\ge \frac{\cos(\pi p_i)}{\sin(\pi \beta_i/2) \sin(\pi \alpha_i/2)}+2\\
&\ge \frac{\cos(\pi p_i)}{\sin^2(\pi p_i/2)}+2\\
&= \frac{2\cos(\pi p_i)}{1- \cos(\pi p_i)}+2\\
&= \frac{2}{1- \cos(\pi p_i)}
\end{align*}

Solving for $p_i$, we obtain
$$\arccos \left(\frac{e^{d_i}-1}{e^{d_i}+1}\right) \le \pi p_i.$$
Summing over $i$, we get
$$\sum_{i=1}^{k}\arccos \left(\tanh\left(\frac{d_i}{2}\right)\right) \le \frac{\pi}{2}.$$

The last part of the theorem is easy to see, for if $d_i < \log\left(\dfrac{1+\cos\frac{\pi}{2k}}{1-\cos\frac{\pi}{2k}}\right)$ for all $i = 1, \cdots, k$, then $\displaystyle\sum_{i=1}^{k}\arccos \left(\tanh\left(\frac{d_i}{2}\right)\right) > \frac{\pi}{2}$, which is a contradiction.
\end{proof}

Now we turn to the proof of Lemma \ref{estimates_lem}. The following elementary lemma is needed.
\begin{lem}[\cite{CS}, Lemma 5.4]\label{measure_theory_lemma}
	Let $(X,\mathscr{B})$ be a measure space and let $\mu$ and $\mu_0$ be two finite measures such that $0 \le \mu_0 \le \mu$. Let $C$ be a Borel set such that $\mu(C) \ge \mu_0(X)$. Let $f$ be a measurable, nonnegative real-valued function on $X$ such that $\text{inf}f(C) \ge \text{sup}f(X-C)$. Then $\displaystyle\int_X f d\mu_0 \le \displaystyle\int_C f d\mu$.
\end{lem}

\begin{proof} [Proof of Lemma \ref*{estimates_lem}]
	Let $h = \text{dist}(z, \gamma z)$ and set $s = \sinh(h)$ and $c = \cosh(h)$. Identify $\overline{\bbH^2}$ with the unit disk so that $z$ gets mapped to $0$ and $\gamma z$ gets mapped to the positive $y$-axis. The Poisson kernel is given by $$P(\phi) = P(z, \gamma^{-1}z,\zeta) = (c-s\cos\phi)^{-1}$$ where $\phi$ is the angle between the vertical axis and the ray from $0$ to $\zeta \in S_{\infty}^1$.
	
	Set $A = A_z$. Since $S_{\infty}^1$ has the round metric centered at $z$, the measure
	$dA = \frac{1}{\pi} d\phi$.
	
	Let $C \subset S_{\infty}^1$ be the arc defined by the inequality $\phi < \pi a$. Then we have $$A(C) = \frac{1}{\pi}\int_0^{\pi a} d\phi = a.$$
	
	Note that $P$ is positive and monotone decreasing for $0 \le \phi \le \pi$, we have $\text{inf}P(C) \ge \text{sup}P(S_{\infty}^1-C)$. Then by Lemma \ref*{measure_theory_lemma},
	\begin{align*}
		b \le \int_{S_{\infty}^1}P d\nu \le \int_C P dA &= \frac{1}{\pi}\int_{0}^{\pi a} \frac{1}{c-s\cos(\phi)}d\phi\\
		&= \frac{2}{\pi} \arctan \left(e^h \tan\left(\frac{\pi a}{2}\right)\right).
	\end{align*}
	Hence, 
	$$e^h \ge \frac{\tan\left(\pi b/2\right)}{\tan\left(\pi a/2\right)}.$$
\end{proof}

\section{The general case}\label{section_general}
Our goal for this section is to generalize Theorem \ref{finite_area_thm} to all free Fuchsian groups. In particular, we prove the Main Theorem stated in Introduction.
\begin{thm}\label{thm_general}
Let $k \ge 2$ be an integer and let $\Phi$ be a free Fuchsian group generated by $\{g_1 \cdots g_k \}$. Let $z$ be any point in $\bbH^2$ and denote by $d_i = \text{d}(z, g_iz)$ the hyperbolic distance between $z$ and $g_iz$. Then we have
\begin{align}\label{ineq_intro}
\sum_{i=1}^k \arccos \left(\tanh\left(\frac{d_i}{2}\right)\right)
&\le \frac{\pi}{2}.
\end{align}
In particular, there exists an $i \in \{1,\cdots,k\}$ such that $$d_i \ge \log\left(\dfrac{1+\cos\frac{\pi}{2k}}{1-\cos\frac{\pi}{2k}}\right).$$ When $k=2$, the upper bound is realized by $\Phi = \Gamma(2) = \left\langle \left( \begin{array}{cc}
1 & 2 \\
0 & 1 \end{array} \right), \left( \begin{array}{cc}
1 & 0 \\
2 & 1 \end{array} \right) \right\rangle$ and $z = \sqrt{-1}$.
\end{thm}

\begin{rmk}\label{rmk_compare}
	Theorem \ref{thm_general} is strictly stronger than the $\log (2k-1)$ Theorem applied to Fuchsian groups. To see this, let $y_i = \arccos\left(\tanh\left(d_i/2\right)\right), i = 1, \cdots, k$. Then $d_i = \log \left(\dfrac{1+\cos y_i}{1-\cos y_i}\right)$ and $$\displaystyle\sum_{i=1}^{k}\dfrac{1}{1+e^{d_i}} \le \frac{1}{2} \text{~if and only if~} \displaystyle\sum_{i=1}^{k} (1-\cos y_i) \le 1.$$
	
	Now assume Theorem \ref{thm_general}, i.e. $\sum_{i=1}^{k}y_i \le \frac{\pi}{2}$. Then $\sum_{i=1}^{k}(1-\cos y_i)$ attains its maximum value $1$ when one of the $y_i$ equals $\pi/2$ and the rest all equal to $0$. For another direction, suppose $\sum_{i=1}^{k} (1-\cos y_i) \le 1$, then letting $y_1 = \cdots = y_k = \arccos(1-1/k)$ gives $\sum_{i=1}^{k} y_i = k\arccos(1-1/k) > \pi/2$ for all $k \ge 2$.	
\end{rmk}

To prove the theorem, we need the following crucial lemma.

\begin{lem} \label{lem_holemb}
Any Riemann surface of finite topological type with negative Euler characteristic admits a holomorphic embedding into a Riemann surface of finite hyperbolic area. Moreover, this embedding is a homotopy equivalence.
\end{lem}
\begin{proof}
Let $S$ be a Riemann surface of finite topological type with negative Euler characteristic. We first show that a funnel end $\mathcal{E}$ of $S$ is conformal to an annulus. To see this, consider the action of $\pi_1(S)$ on the upper-half space $\bbH^2$. Up to conjugation, we can assume that the element $\gamma \in \pi_1(S)$ representing the end fixes the positive imaginary axis. Then $\gamma(z) = az$ for some $a>1$ and $\mathcal{E} = R / \langle \gamma \rangle$, where $R=\{z \in \bbH^2 ~|~ Im(z)>0 {\text ~and~} Re(z) \ge 0\}$. Consider the map $\phi(z) = e^{-2\pi i \log z/ \log a}$. $\phi$ is holomorphic in $R$ and $\phi(R)$ is the annulus $A = \{1< |\phi| < e^{\pi^2/\log a}\}$. Since $\phi$ is $\gamma$ invariant, $\mathcal{E}$ is confomal to $A$. Therefore, $S$ is conformal to the surface $S'$ obtained by attaching a finite Euclidean cylinder $S^1 \times (0,t)$ to each boundary geodesic of its convex core $C(S)$. 

Consider the Riemann surface $S_{\infty}$ obtained by attaching a half infinite Euclidean cylinder $S^1 \times \bbR^+$ to each boundary geodesic of $C(S)$. We claim that $S_{\infty}$ has finite area. To see this, first note that if $\gamma$ is a geodesic in the homotopy class of a finite cylinder $A$, then its hyperbolic length $\ell(\gamma)$ satisfies $\ell(\gamma) < \pi/(\text{modulus of $A$})$. Since a half-infinite cylinder contains cylinders of arbitrarily large moduli, the hyperbolic length of the geodesic $\gamma$, if it exists, would satisfy $\ell(\gamma) < \varepsilon$ for any $\varepsilon >0$. Hence, there are no geodesics in the homotopy classes of glued cylinders, i.e. the ends of $S_{\infty}$ are cusps.

Clearly, $S'$ is embedded in $S_{\infty}$ and the embedding is a homotopy equivalence.
\end{proof}

\begin{proof}[Proof of Theorem 4.1]
Suppose the quotient surface $\Sigma = \bbH^2/\Phi$ has infinite area. Fix $z \in \bbH^2$. We can assume that the projection of $z$ on $\Sigma$ lies in the convex core $C(\Sigma)$ because otherwise we can apply the nearest point retraction map $R : \Sigma \to C(\Sigma)$ to project $z$ and $g_i(z)$ onto the boundary $\partial C(\Sigma)$ which decreases the distance $d(z, g_i(z))$. Then for $i =1, \cdots, k$, the projections of the geodesics connecting $z$ to $g_i(z)$ lie in $C(\Sigma)$.
 
By Lemma \ref{lem_holemb}, there is a holomorphic embedding of $\Sigma$ into a finite area surface $\Sigma_{\infty}$. Since holomorphic maps are distance non-increasing with respect to the hyperbolic metric, we get $d_{\Sigma} (z, g_i(z)) \ge d_{\Sigma_{\infty}} (z, g_i(z))$. Since $\arccos(\tanh(x/2))$ is strictly decreasing, we have
\begin{align*}
\sum_{i=1}^k \arccos \left(\tanh\left(\frac{d_{\Sigma} (z, g_i(z))}{2}\right)\right)
& \le \sum_{i=1}^k \arccos \left(\tanh\left(\frac{d_{\Sigma_{\infty}} (z, g_i(z))}{2}\right)\right)\\
& \le \frac{\pi}{2}.
\end{align*}
The second inequality follows as $\Sigma_{\infty}$ has finite area.

Straightforward calculation shows that when $k=2$ the upper bound is realized by the given group and the point.
\end{proof}

\section{Quantitative geometry of hyperbolic surfaces}\label{section_app}
In this section, we discuss applications of Theorem \ref{thm_general} to quantitative geometry of hyperbolic surfaces.

\subsection{The two-dimensional Margulis constant}
Let $M$ be a closed orientable hyperbolic $n$-manifold. A positive real number $\varepsilon$ is called a {\it Margulis number} for $M$ if for any $z \in \bbH^n$ and $\alpha, \beta \in \pi_1M$, $\text{max}\{d(z,\alpha(z)), d(z,\beta(z))\} < \varepsilon$ implies that $\alpha$ and $\beta$ commute. The Margulis Lemma (\cite{BP}, Chapter D) states that for every $n \ge 2$ there is a positive constant which is a Margulis number for every closed orientable hyperbolic $n$-manifold. The largest such constant $\varepsilon_n$ is called the {\it Margulis constant} for closed orientable hyperbolic $n$-manifold.

An immediate corollary of Theorem \ref{thm_general} recovers the following result due to Yamada.
\begin{cor}[Yamada, \cite{Yamada}]
	The two-dimensional Margulis constant $\varepsilon_2$ equals $\log(3+2\sqrt{2})$.
\end{cor}
\begin{proof}
	Let $\Gamma$ be a discrete group of orientation-preserving isometries of $\bbH^2$. Then $\varepsilon$ is a Margulis number for $\bbH^2/\Gamma$ if and only if for any non-commuting pair $\alpha, \beta \in \Gamma$ and every point $z \in \bbH^2$, at least one of $\alpha$ and $\beta$ moves the point $z$ at least $\varepsilon$ distance away from itself. Moreover, the group generated by $\alpha$ and $\beta$ is free of rank $2$. Theorem \ref{thm_general} (with $k=2$) gives the desired Margulis number. Furthermore, this number is realized and therefore it is the Margulis constant.
\end{proof}

\subsection{Lengths of a pair of closed loops}
In this subsection, we discuss the consequences of Theorem \ref{thm_general} on lengths of a pair of based loops on a hyperbolic surface.

\begin{cor} 
If $\alpha_1$ and $\alpha_2$ are two loops on a hyperbolic surface $\Sigma$ based at a point $p$ such that they define non-commuting elements in the fundamental group $\pi_1(\Sigma, p)$, then $$\sinh(\ell_1/2)\cdot \sinh(\ell_2/2) \ge 1$$ where $\ell_i$ denotes the hyperbolic length of $\alpha_i, i=1,2$. 
\end{cor}

\begin{proof}
Let $z$ be a lift of the base point $p$ in $\bbH^2$. Since $[\alpha_1], [\alpha_2] \in \pi_1\Sigma$ generate a free subgroup of rank $2$ and $d(z, [\alpha_i]z) = \ell_i$, by Theorem \ref{thm_general}, $\sum_{i=1}^2 \arccos(\tanh(\ell_i/2)) \le \pi/2.$

Applying cosine function on both sides of the inequality, we get
$$\tanh(\ell_1/2)\cdot \tanh(\ell_2/2) - \frac{1}{\cosh(\ell_1/2)\cdot \cosh(\ell_2/2) } \ge 0$$ and the desired inequality follows.
\end{proof}

The above corollary improves a result of Buser's (\cite{Bu}) which gives the same inequality but requires $\alpha_1$ and $\alpha_2$ to be two intersecting simple closed geodesics.

\end{document}